\documentclass[11pt,a4paper]{article}
\usepackage[T1]{fontenc}
\usepackage[utf8]{inputenc}
\usepackage{amsmath}
\usepackage{amssymb}
\usepackage{color}
\usepackage{amsthm}
\usepackage{verbatim}
\usepackage{mathabx}
\usepackage{comment}
\usepackage[pagebackref]{hyperref}
\usepackage{mathrsfs}
\usepackage{calligra}
\usepackage[shortlabels]{enumitem}
\usepackage[refpage, intoc]{nomencl}

\makenomenclature
\usepackage{a4wide}
\newtheorem{definition}{Definition}[section]
\newtheorem{theorem}[definition]{Theorem}
\newtheorem{corollary}[definition]{Corollary}
\newtheorem{remark}[definition]{Remark}
\newtheorem{proposition}[definition]{Proposition}
\newtheorem{lemma}[definition]{Lemma}

\newtheorem*{theorem*}{Theorem}

\newcommand{\ii}{\ensuremath{\imath}}
\newcommand{\RR}{\ensuremath{\mathbb{R}}}

\newcommand{\dd}{\mathrm{d}}
\newcommand{\supp}{\ensuremath{\mathrm{supp}\,}}
\newcommand{\EE}{\mathbb{E}}

\newcommand{\D}{\nabla}
\newcommand{\Wick}[1]{\mathbf{:}#1\mathbf{:}}
\newcommand{\wDY}{\Wick{\D Y^2}}
\newcommand{\wDYe}{\Wick{\D Y^2_\varepsilon}}
\newcommand{\tDY}{\widetilde{\wDY}}
\newcommand{\tDYe}{\widetilde{\wDYe}}
\newcommand{\w}{\langle x\rangle}

\newcommand{\R}[1]{\mathrm{Re}\left( #1\right)}
\newcommand{\I}[1]{\mathrm{Im}\left( #1 \right)}
\newcommand{\vk}{\ensuremath{v_{2^{-k}}}}
\newcommand{\vkp}{\ensuremath{v_{2^{-k-1}}}}
\newcommand{\Yk}{Y_{2^{-k}}}
\newcommand{\Ykp}{Y_{2^{-k-1}}}

\title{Solution to the stochastic Schr\"odinger equation on the full space}
\author{
  Arnaud Debussche\thanks{Partially supported by the French government thanks to the
``Investissements d'Avenir" program ANR-11-LABX-0020-01.}\\
  IRMAR, ENS Rennes, UBL, CNRS \\
\and 
Jörg Martin\thanks{Financial support by the DFG via Research Training Group RTG 1845 is gratefully acknowledged.} \\
  Humboldt--Universit\"at zu Berlin 
  }

\begin{document}
\maketitle

\begin{abstract}
We here show how the methods recently applied in \cite{SchroedingerTorus} to solve the  stochastic nonlinear Schrödinger equation on $\mathbb{T}^2$ can be enhanced to yield solutions on $\RR^2$ if the non-linearity is weak enough. We prove that the solutions remains localized on compact time intervals which allows us to apply energy methods on the full space. 
\end{abstract}
\section{Introduction}
We here study the following Cauchy problem on $[0,T]\times \RR^2$
\begin{align}
\label{eq:NLS}
\ii \partial_t u =\Delta u+\lambda u|u|^{2\sigma} + u\xi, \, u(0)=u_0
\end{align} 
where $\sigma>0,\,\lambda\in \RR$ and $\xi\in \mathcal{S}'(\RR^2)$ stands for white noise in space. \eqref{eq:NLS} can be seen as stochastic version of the well-studied  deterministic nonlinear Schrödinger equation (for example \cite{Cazenave}, \cite{Kato}, \cite{Cazenave1979}, \cite{BrezisGallouet}), where the noise term $u\xi $ is absent. The case $\sigma<1$ is known as the subcritical regime and $\sigma\geq 1$ as (super-)critical. For a positive coefficient $\lambda>0$ the equation is called focusing, while for $\lambda<0$ one uses the term defocusing. 

While the deterministic equation arises in nonlinear optics to model laser propagation in a dispersive material \cite[Section 1.1.-1.3.]{Berge}, the stochastic term can be seen as taking into account disorder in the considered medium. In the physical context interesting effects such as Anderson localization and the existence of solitons have been observed \cite{AndersonLocalization, Conti, ContiFolli}.

The existence of solutions for \eqref{eq:NLS} with $\sigma=1$ on a periodic setup $[0,T]\times \mathbb{T}^2$ ($\mathbb{T}^2$ denoting the torus) was recently studied in \cite{SchroedingerTorus}, where the authors prove global existence of solutions. However, most of physical phenomena described by \eqref{eq:NLS} are set on the full space $\RR^2$. 
For instance, if one is interested in solitary waves and their stability, it is important to have a notion for a solution on $\RR^2$. This question is the main motivation for this work and will
be the object of further studies.  

In this article we show that the equation does indeed possess solutions on the full space if the non-linearity is not too strong. Our main results can be summarized as follows.  

\begin{theorem}
Under suitable initial conditions $u_0$ \eqref{eq:NLS} has a unique local solution on $[0,T]\times \RR^2$ for some random time $T>0$ if $\lambda\leq 0$ or $\sigma<1$.  If $\sigma<1/2$ (or $\lambda=0$) the solution is global. 
\end{theorem} 
For a more precise formulation see Theorems \ref{thm:LocalExistence} and  \ref{thm:GlobalExistence} below.

It might seem surprising that even for local well-posedness the sign of $\lambda$ matters. Recall that in the absence of noise, local well-posedness is known in the energy space 
thanks to the Strichartz estimates. We do not know whether similar estimates hold for the noisy equation considered here. 
As in \cite{SchroedingerTorus} we rely on a transformation of the equation that stems originally from \cite{HairerLabbe}  and make use of conservation of mass and energy to prove a priori bounds for \eqref{eq:NLS}. An application of these methodes is more delicate on the full space since one needs controll over the growth of the solution to counterbalance the noise. A key role in this task will be played by Lemma \ref{lem:H1GivesMoments} below which allows us to trade some differentiability of the solution against some localization on compact time intervals. 

Our analysis will be formulated in the language of weighted Besov spaces whose definition and fundamental properties we recall in Section \ref{sec:DefinitionsAndBasics} below. We then describe the growth behaviour of the noise in our equation in this framework and finally give a rigorous formulation of the problem we solve. Roughly speaking, Section \ref{sec:H1} is devoted to the controll of moments, while \ref{sec:LocalExistence} and \ref{sec:GlobalExistence} show an $H^2$ bound which allows for a solution of \eqref{eq:NLS} in Theorems \ref{thm:LocalExistence} and \ref{thm:GlobalExistence}.

\subsection*{Notation}
We write $C>0$ for deterministic constants and $K_\varepsilon$ for random constants depending on $\varepsilon\in (0,1]$ with $L^p(\mathbb{P}),\,p\in [1,\infty)$ norms bounded independent of $\varepsilon$. We also write $K_k=K_{2^{-k-1}} K_{2^{-k}}$ and indicate further by this notation that this sequence is bounded almost surely in $k=0,1,2,\ldots\,$. We sometimes use $a>0$ for a positive, deterministic constant appearing in the exponent. The constants $C,\,K_\varepsilon,\,K_k$ and $a$ may change from one line to another.

An important role in this paper will be played by the polynomial weights
\begin{align*}
	\w^\mu:=(1+|x|^2)^{\mu/2}\,,
\end{align*}
with $\mu\in \RR$. For weighted $L^p$ spaces $L^p(\RR^d,\w^\mu)=L^p(\w^\mu),\,p\in [1,\infty],\,\mu\in \RR$ we here use the convention
\begin{align*}
\|f\|_{L^p(\w^\mu)}:=\|\w^\mu f\|_{L^p(\RR^d)}=\left( \int_{\RR^d} |f(x)\w^\mu|^p \, \dd x \right)^{1/p}\,.
\end{align*}
(with the usual interpretation if $p=\infty$). Most of the further notation and used function spaces are described in subsection  \ref{subsec:BesovSpaces}. 
\section{Definitions and Basics}
\label{sec:DefinitionsAndBasics}
\subsection{Weighted Besov and Sobolev spaces}
\label{subsec:BesovSpaces}
We follow here closely \cite{TriebelIII} and \cite{Bahouri} and refer to these books for details. The definition of weighted Besov spaces used in this article is
\begin{align*}
\mathcal{B}^{\alpha}_{p,q}(\RR^d,\w^\mu)=\mathcal{B}^{\alpha}_{p,q}(\w^\mu)=\left\{f\in \mathcal{S}'(\RR^d)\,\vert\, \|f\|_{\mathcal{B}^\alpha_{p,q}(\w^\mu)}:=\|2^{j\alpha } \|\varDelta_j f\|_{L^p(\w^\mu)}\|_{\ell^q}<\infty \right\}\,,
\end{align*}
where $(\varDelta_j)_{j=-1,0,\ldots}$ denotes some choice of Littlewood-Paley-blocks and where $p,q\in [1,\infty],\,\alpha,\mu\in \RR$. We will write $\mathcal{B}^{\alpha}_{p,q}=\mathcal{B}^{\alpha}_{p,q}(\RR^2)$ for the unweighted space $\mathcal{B}^{\alpha}_{p,q}(\w^0)$. A nice property of these spaces is that the weight can be "pulled in",
\begin{align}
\label{eq:PullWeight}
\|f\|_{\mathcal{B}^{\alpha}_{p,q}(\w^\mu)} \approx \|f \w^\mu \|_{\mathcal{B}^{\alpha}_{p,q}}\,,
\end{align}
in the sense of equivalent norms, for $\alpha,\,\mu\in \RR,\,p,q\in [1,\infty]$, compare \cite[Theorem 6.5]{TriebelIII}, a fact we will use throughout this paper to translate well-known results from the unweighted spaces to their weighted analogues. Among these are the following immediate consequences which we mention here for later usage.
\begin{lemma}\label{lem:BesovProperties}
The weighted Besov spaces defined above have the following properties:
\begin{enumerate}[(i)]
\item (Besov embedding)
For $\alpha_1,\,\alpha_2,\,\mu_1,\,\mu_2\in \RR$ and $p_1,\,p_2,\,q_1,\,q_2\in [1,\infty]$ with $\mu_1\leq \mu_2$, $\alpha_1-\frac{d}{p_1}\leq \alpha_2-\frac{d}{p_2}$, $q_2\leq q_1$ and $p_2\leq p_1$, we have the continuous embedding 
\begin{align*}
\mathcal{B}^{\alpha_2}_{p_2,q_2}(\RR^d,\w^{\mu_2})\subseteq  \mathcal{B}^{\alpha_1}_{p_1,q_1}(\RR^d,\w^{\mu_1})\,.
\end{align*}
\item (Sobolev embedding) For $\alpha>0,\, p\in [2,\infty]$ such that $-\frac{d}{p}\leq\alpha-\frac{d}{2}$ and $\mu_1,\mu_2\in \RR,\,\mu_1 \leq \mu_2$ we have the continuous embedding
\begin{align*}
\mathcal{B}^{\alpha}_{2,2}(\w^{\mu_2}) \subseteq L^p(\w^{\mu_1})\,.
\end{align*}
\item (Duality)
For $\alpha\in \RR,\,p,q\in [1,\infty),\,\mu\in \RR$ we have the duality
\begin{align*}
\left(\mathcal{B}^{\alpha}_{p,q}(\RR^d,\w^{\mu})\right)'= \mathcal{B}^{-\alpha}_{p',q'}(\RR^d,\w^{-\mu})\,.
\end{align*}
\item (Multiplication)
For $\mu_1,\,\mu_2\in \RR$, $p_1,\,p_2\in [2,\infty]$ and $\alpha_1,\,\alpha_2>0$ we have for $\alpha=\alpha_1 \wedge \alpha_2$, $\frac{1}{p}=\frac{1}{p_1}+\frac{1}{p_2}$ and $\mu=\mu_1+\mu_2$
\begin{align*}
\Big\|f_1\cdot f_2\|_{\mathcal{B}^{\alpha}_{p,p}(\w^\mu)}\leq C \|f_1\|_{\mathcal{B}^{\alpha_1}_{p_1,p_1}(\w^{\mu_1})} \|f_2\|_{\mathcal{B}^{\alpha_2}_{p_2,p_2}(\w^{\mu_2})}\,.
\end{align*}
\end{enumerate}
\end{lemma}
\begin{remark}
Property \textit{(iv)} is a special case of the far more general paraproduct inequalities. We refer to \cite{Bahouri} or, for a slightly more general version (which was applied here), to \cite{PromelTrabs}. Property \textit{(i)} and \textit{(iii)} are standard \cite[Theorem 2.7.1, Theorem 2.11.2]{TriebelI}, property \textit{(ii)} is probably also standard but is also a direct consequence of \cite[Theorem 6.7, 6.9]{TriebelIII}.
\end{remark}
Another fact we will use extensively is the interpolation between weighted Besov spaces with mixed weights, which we take from \cite[Theorem 4.5]{Sickel}.
\begin{lemma}\label{lem:BesovInterpolation}
Let $p_0,q_0,p_1,q_1\in [1,\infty],\,\alpha_0,\alpha_1,\mu_0,\mu_1\in \RR$ with $p_0\vee q_0<\infty$ or $p_1\vee q_1<\infty$. Let further $p,q,\alpha,\mu$ be such that there is a $\varTheta\in [0,1]$ with $\frac{1}{p}=\frac{1-\varTheta}{p_0}+\frac{\varTheta}{p_1},\,\frac{1}{q}=\frac{1-\varTheta}{q_0}+\frac{\varTheta}{q_1},\,\alpha=(1-\varTheta)\alpha_0+\varTheta \alpha_1$ and $\mu=(1-\varTheta)\mu_0+\varTheta \mu_1$. It holds
\begin{align*}
\|f\|_{\mathcal{B}^{\alpha}_{p,q}(\w^\mu)}\leq C \|f\|^{1-\varTheta}_{\mathcal{B}^{\alpha_0}_{p_0,q_0}(\w^{\mu_0})} \|f\|^\varTheta_{\mathcal{B}^{\alpha_1}_{p_1,q_1}(\w^{\mu_1})}\,.
\end{align*} 

\end{lemma}

We will frequently use the following identities for $\alpha,\,\mu\in \RR$
\begin{align*}
\mathcal{B}^{\alpha}_{2,2}(\w^\mu) &=H^\alpha(\w^\mu)\,,\\
\mathcal{B}^{\alpha}_{\infty,\infty}(\w^\mu) &=\mathcal{C}^{\alpha}(\w^\mu)\,.
\end{align*}
where $H^\alpha(\w^\mu)$ is the (weighted) Bessel-potential space given by those $f$ for which $\|\mathcal{F}^{-1}\langle\cdot\rangle^\alpha \mathcal{F}f\|_{L^2(\w^\mu)}<\infty$ and where $\mathcal{C}^\alpha(\w^\mu)$ is the (weighted) Hölder-Zygmund space that coincides with the classical Hölder space for $\alpha\in \RR_{+}\backslash \{0,1,2,\ldots\}$ with the equivalent norm
\begin{align}
\label{eq:ZygmundNorm}
\sum_{|k|\leq \lfloor \alpha\rfloor}\sup_{x} \,\w^{\mu}\,|\partial^k f(x)|+\sum_{|k|=\lfloor \alpha \rfloor }\sup_{0<|x-y|\leq 1}\w^\mu\frac{|\partial^k f(x)-\partial^k f(y)|}{|x-y|^{\alpha -\lfloor \alpha\rfloor}}\,.
\end{align}
We refer to \cite[Section 6.1, 6.2]{TriebelIII} for these identities. Restriction in \eqref{eq:ZygmundNorm} to $x,y\in \Omega$ and setting $\mu=0$ gives rise to the local H\"older space $\mathcal{C}^\alpha(\Omega)$. We can then bound \eqref{eq:ZygmundNorm} by 
\begin{align}
\label{eq:BoundZygmundByChunks}
\|f\|_{\mathcal{C}^{\alpha}(\w^\mu)}\leq C \sup_{k\in \mathbb{N}} \,\langle k\rangle^\mu \, \|f\|_{\mathcal{C}^{\alpha}([-k,k]^2)}
\end{align}
for $\mu\in \RR$ (with equivalence if $\mu\leq 0$).

\subsection{Growth of $\xi$ and $Y$}
\label{subsec:Noise}
We recall the following result on the growth of the white noise, which we take from \cite[Lemma 5.3]{ChoukAllez}.
\begin{lemma}
\label{lem:GrowthNoise}
Let $\xi$ be white noise on $\RR^2$ and pick uniformly bounded $\chi_k\in C^2(\RR^2)$ with $\supp\,\chi_k \subseteq [-k-1,k+1]^2$ and $\chi_k\vert_{[-k,k]^2}=1 $. For $\alpha<1$ there are $\lambda,\lambda'>0$ such that
\begin{align*}
	\sup_{k\in \mathbb{N}} \frac{\EE\left[\exp( \lambda \|\xi \chi_k\|^2_{{\mathcal{C}^{\alpha-2}(\RR^2)}})\right]}{k^{\lambda'}}<\infty\,.
\end{align*}
\end{lemma}
\begin{remark}
In \cite{ChoukAllez} the authors actually bound periodic white noise $\xi_k$ on $[-k,k]^2$ instead of $\chi_k\xi$. However this can be easily translated into the result above due to $\xi_{k+1}\chi_k\overset{\dd}{=}\xi \chi_k$ and $\|\chi_k\xi_{k+1}\|_{\mathcal{C}^{\alpha-2}(\RR^2)} \leq C \|\xi_{k+1}\|_{\mathcal{C}^{\alpha-2}([-k,k]^2 )}$ for $\alpha\in (0,1)$.
\end{remark}
We now proceed as in \cite{HairerLabbe} and use a truncated Green's function $G\in C^\infty(\RR^2\backslash\{0\})$ that satisfies $\supp G \subseteq B(0,1)$ and $G(x)=\frac{1}{2\pi} \log|x|$ for $|x|$ small enough, so that $Y:=G\ast \xi$ solves
\begin{align*}
\Delta Y=\xi+\varphi\ast \xi
\end{align*} 
for some $\varphi\in C^\infty_c(\RR^2)$. 
We recall the following result from \cite{HairerLabbe}.
\begin{lemma}
For any $p\in [1,\infty),\,\delta>0$ and $\alpha\in (0,1)$ we have
\begin{align*}
	\EE\left[\|Y\|_{\mathcal{C}^\alpha(\w^{-\delta})}^p+\|\wDY\|_{\mathcal{C}^{\alpha-1}(\w^{-\delta})}^p\right]<\infty\,.
\end{align*}
\end{lemma}
With the help of Lemma \ref{lem:GrowthNoise} we also deduce a bound on $e^{Y}$.
\begin{corollary}\label{cor:BoundeY}
For any $a\in \RR,\,\alpha\in (0,1),\delta>0$ and $p\in [1,\infty)$ we have 
\begin{align*}
\EE\left[\|e^{aY}\|^p_{\mathcal{C}^{\alpha}(\w^{-\delta})}\right]<\infty\,.
\end{align*}
\end{corollary}
\begin{proof}
Note first that we can bound 
\begin{align*}
	\|e^{aY}\|_{\mathcal{C}^{\alpha}(\w^\delta)}\overset{\eqref{eq:BoundZygmundByChunks}}{\leq} C \sup_{k\in \mathbb{N}} \frac{\|e^{aY}\|_{\mathcal{C}^{\alpha}([-k,k]^2)}}{k^\delta} \leq C \sup_{k\in \mathbb{N}} \frac{\exp(C|a|\|Y\|_{\mathcal{C}^{\alpha}([-k,k]^2)})}{k^\delta}\,,
\end{align*}
so that it remains to bound the $p$-th moment of the right hand side. 
Using the compact support of the Green's function we see 
\begin{align*}
\|Y\|_{\mathcal{C}^{\alpha}([-k,k]^2)}\leq C \|\chi_{k+2}\xi\|_{\mathcal{C}^{\alpha-2}(\RR^2))}
\end{align*}
(to show this one can for example use the wavelet characterization of Besov spaces, as in \cite{TriebelIII}, and a decomposition of $G$ as in \cite[Remark 5.6]{RegularityStructures}). With Lemma \ref{lem:GrowthNoise} we can therefore find $\lambda,\,\lambda'>0$ such that
\begin{align*}
	\sup_{k\in \mathbb{N}} \frac{\EE\left[ \exp(\lambda \|Y\|^2_{\mathcal{C}^{\alpha}([-k,k]^2)}) \right] }{k^{\lambda'}}<\infty\,.
\end{align*}
We now pick, without loss of generality, $p\in [1,\infty)$ so big that $p \cdot \delta \geq 2+\lambda'$ which gives us
\begin{align*}
\EE\left[\left|\sup_{k\in \mathbb{N}} \frac{\exp(C|a|\|Y\|_{\mathcal{C}^{\alpha}([-k,k]^2)})}{k^\delta}\right|^p\right] &\leq \sum_{k=1}^\infty \frac{\EE\left[\exp(pC|a|\|Y\|_{\mathcal{C}^\alpha([-k,k]^2)}\right]}{k^{\delta p}} \\
&\leq C \sum_{k=1}^\infty \frac{\EE\left[\exp(\lambda\|Y\|^2_{\mathcal{C}^\alpha([-k,k]^2)}\right]}{k^{2}\cdot k^{\lambda'}}<\infty\,.
\end{align*}
\end{proof}
By similar arguments as in Corollary \ref{cor:BoundeY} a bound for $\varphi\ast \xi$ follows 
\begin{align*}
\EE\left[\|\varphi\ast \xi \|_{\mathcal{C}^{\beta}(\w^{-\delta})}^p \right]<\infty
\end{align*}
for any $\beta\in \RR$, $\delta>0$ and $p\in [1,\infty)$.
We will mostly work with smoothened noise, so fix a mollifier $\rho\in C^\infty_c(B(0,1)$ and define for $\rho_\varepsilon:=\varepsilon^{-2}\,\rho(\varepsilon^{-1}\cdot)\,\varepsilon\in (0,1]$ 
\begin{align*}
\xi_\varepsilon=\rho_\varepsilon\ast \xi,\, Y_\varepsilon=G\ast \xi_\varepsilon\,.
\end{align*}
and we have once more (with the same $\varphi\in C^\infty_c(\RR^2)$ as above)
\begin{align*}
\Delta Y_\varepsilon=\xi_\varepsilon+\varphi \ast\xi_\varepsilon \,.
\end{align*}
Using that $\supp \rho_\varepsilon \subseteq B(0,1)$ we can redo all the proofs above and obtain that for any $\delta>0,\alpha\in (0,1),\,\beta\in \RR,\,a\in \RR$ and $p\in [1,\infty)$ we have the following bound uniform in $\varepsilon$
\begin{align}
\label{eq:BoundSmoothenedNoise}
\EE\left[\|Y_\varepsilon\|^p_{\mathcal{C}^\alpha(\w^{-\delta})}+\|\wDYe\|^p_{\mathcal{C}^{\alpha-1}(\w^{-\delta})}+\|e^{aY_\varepsilon}\|^p_{\mathcal{C}^\alpha(\w^{-\delta})}+\|\varphi\ast \xi_\varepsilon\|^p_{\mathcal{C}^{\beta}(\w^{-\delta})}\right]\le C\,.
\end{align}
Let us also recall the following statements, again from \cite{HairerLabbe}.
\begin{lemma}
For $\alpha\in (0,1)$, $\kappa\in (0,1-\alpha)$, $p\in [1,\infty)$ we have 
\begin{align*}
\mathbb{E}\left[\|Y_\varepsilon-Y\|^p_{\mathcal{C}^{\alpha}(\w^{-\delta})}+\|\wDYe-\wDY\|^p_{\mathcal{C}^{\alpha-1}(\w^{-\delta})}\right]\leq C \varepsilon^{p\kappa}\,.
\end{align*}
\end{lemma}
Together with the bounds \eqref{eq:BoundSmoothenedNoise} and Corollary \ref{cor:BoundeY} we then obtain 
\begin{align}
\label{eq:ConvergenceeY}
\EE\left[\|e^{aY}-e^{aY_\varepsilon}\|^p_{\mathcal{C}^{\alpha}(\w^\delta)}\right]\leq C \varepsilon^{p\kappa }
\end{align}
for $a\in \RR,\,\alpha\in (0,1)$ and $\kappa\in (0,1-\alpha)$. Further we have for $\beta\in\RR,\,\alpha\in (0,1)$ $\delta>0$, $p\in [1,\infty)$ and $\kappa\in (0,1-\alpha)$
\begin{align}
\label{eq:ConvergenceSmoothTerm}
\EE\left[\|\varphi\ast \xi_\varepsilon-\varphi\ast \xi\|^p_{\mathcal{C}^{\beta}(\w^{-\delta})}\right]\leq C \EE\left[\|\xi-\xi_\varepsilon\|_{\mathcal{C}^{\alpha}(\w^{-\delta})}\right]\leq C \varepsilon^{\kappa}
\end{align}
where we used in the last step $\|\xi-\xi_\varepsilon\|_{\mathcal{C}^{\alpha}(\w^{-\delta})}\leq C\varepsilon^\kappa$ due to \cite[Lemma 1.1]{HairerLabbe}.
It will turn out convenient to have an  estimate on the blow-up of the $L^p$ norm of $\D Y$, which is covered by the following Lemma. 
\begin{lemma}\label{lem:NoiseLp}
 Given $\delta\in (0,1)$, $p\in (2/\delta,\infty)$ and $q\in [p,\infty)$ we have
\begin{align*}
\mathbb{E}\left[ \|\D Y_\varepsilon\|_{L^p(\w^{-\delta})}^q+\|\wDYe\|_{L^p(\w^{-\delta})}^q\right]\leq C |\log(\varepsilon)|^q
\end{align*}
\end{lemma} 
\begin{proof}
We estimate via Jensen's inequality
\begin{align*}
\mathbb{E}\left[ \|\D Y_\varepsilon\|_{L^p(\w^{-\delta})}^q+\|\wDYe\|_{L^p(\w^{-\delta})}^q\right]\leq C\int_{\RR^2} \frac{1}{\w^{p\delta}}\,\mathbb{E}[|\D Y(x)|^q+|\wDYe(x)|^q]\, \dd x\,.
\end{align*}
The result follows now from equivalence of moments and $\mathbb{E}[|\D Y_\varepsilon(x)|^2]$, $\EE[\wDYe(x)]\leq C |\log \varepsilon|^2$.
\end{proof}
\subsection{Set-up and conserved quantities}
\label{subsec:SetUp}
We want to consider the equation 
\begin{align}
\label{eq:NotTransformedEquation}
\ii \partial_t u=\Delta u+u \xi +\lambda |u|^{2\sigma} u,\,u(0)=u_0\,,
\end{align}
with a suitable renormalization we introduce below. It is well-known (see for example \cite{Cazenave}) that a solution to this equation, if existent, has at least formally the conserved quantities mass $N(u(t))=N(u_0)$ and energy $H(u(t))=H(u_0)$ defined as
\begin{align*}
N(u)=\int_{\RR^2} |u|^2 \,\dd x,\,
H(u)=\int_{\RR^2} \frac{1}{2} |\D u|^2 -\frac{1}{2} |u|^2 \xi -\frac{\lambda}{2\sigma+2} |u|^{2\sigma+2} \,\dd x\,.
\end{align*}
We follow \cite{SchroedingerTorus} in the idea to substitute $u$ in this equation by $v=e^{Y}u$ with $Y$ as introduced above and obtain the problem
\begin{align*} 
\ii \partial_t v=\Delta v+v\,(\D Y^2 -\varphi\ast\xi)-2\D v \D Y+|v|^{2\sigma} v e^{-2\sigma Y} ,\, v(0)=v_0:=e^{-Y} u_0\,.
\end{align*}
As explained in \cite{SchroedingerTorus} there is only hope to obtain a solution to this equation if we replace the square $\D Y^2$ by a different expression for which we take the Wick product $\wDY$, which corresponds to the renormalization announced for \eqref{eq:NotTransformedEquation} above.
\begin{align} 
\label{eq:TransformedEquation}
\ii \partial_t v=\Delta v+v\,(\wDY-\varphi\ast\xi)-2\D v \D Y+\lambda |v|^{2\sigma} v e^{-2\sigma Y},\,v(0)=v_0\,.
\end{align}
We assume for simplicity, as in \cite{SchroedingerTorus}, that the initial condition is ``controlled by $Y$'' in the sense $v_0=u_0 e^{Y}\in H^2(\w^{\delta_0})$ for some $\delta_0\in (0,\frac{1}{2})$. $v$ has then, at least formally, the conserved quantities 
\begin{align}
\label{eq:TransformedQuantities}
\tilde{N}(v)=\int_{\RR^2}  |v|^2 e^{-2Y}\, \dd x ,\,\tilde{H}(v)=\int_{\RR^2 } \left(\frac{1}{2} |\D v|^2 -\frac{1}{2}|v|^2 \tDY  -\frac{\lambda}{2\sigma+2}|v|^{2\sigma+2}\right) e^{-2Y}\,\dd x\,,
\end{align}
where we introduced the notation $\tDY=\wDY-\varphi\ast \xi$.
Our aim is to solve \eqref{eq:TransformedEquation} by approximation via a smoothened equation
\begin{align}
\label{eq:MollifiedEquation}
\ii \partial_t v_\varepsilon=\Delta v_\varepsilon+v_\varepsilon\tDYe-2\D v_\varepsilon \D Y_\varepsilon+\lambda |v_\varepsilon e^{-Y_\varepsilon}|^{2\sigma } v_\varepsilon ,\,v_\varepsilon(0)=v_0
\end{align}
where $Y_\varepsilon$ is defined as above via the mollification $\xi_\varepsilon=\rho_\varepsilon\ast \xi$ and with $\tDYe=\wDYe-\varphi\ast \xi_\varepsilon$. Equation \eqref{eq:MollifiedEquation} has a (unique) solution for any $T>0$ in $C([0,T];H^{2}(\w^{-\delta}))\cap C([0,T];H^{\gamma}(\w^{\delta'}))$ for any $\delta>0$, $\gamma\in (1,2)$ and $\delta'<(1-\frac{\gamma}{2})\delta_0$, see \cite[Section 3.6]{Cazenave}\footnote{The growth result is not contained in \cite{Cazenave} but follows from the same arguments than below if one first cuts-off the potential $\xi_\varepsilon$, then derives bounds independent of the truncation and finally removes the latter.}.

 \eqref{eq:MollifiedEquation} has the conserved quantities
\begin{align*}
&\tilde{N}(v_\varepsilon)=\int_{\RR^2}|v_\varepsilon|^2 e^{-2Y_\varepsilon} \, \dd x \\ &\tilde{H}_\varepsilon(v_\varepsilon)=\int_{\RR^2} \left(\frac{1}{2}|\D v_\varepsilon|^2-\frac{1}{2}|v_\varepsilon|^2  \tDYe-\frac{\lambda}{2\sigma+2} |v_\varepsilon |^{2\sigma+2} e^{-2\sigma Y_\varepsilon}   \right) e^{-2Y_\varepsilon}\,\dd x\,.
\end{align*}
which are well-defined due to $v_\varepsilon\in C([0,T];H^{\gamma}(\w^{\delta}))$ for $\gamma\in(1,2)$ and $\delta<(1-\frac{\gamma}{2})\delta_0$.

\section{Moments and a priori bound in $H^1$}
\label{sec:H1}
We start by a small lemma that allows us to controll moments of $v_\varepsilon$ by its derivatives. 
\begin{lemma}\label{lem:H1GivesMoments}
Let $v_\varepsilon$ be the unique solution to \eqref{eq:MollifiedEquation} on $[0,T]$. We then have for $\delta\in (0,\delta_0)$ and $\delta'<1-2\delta$
\begin{align*}
\sup_{t\in [0,T]} \int_{\RR^2} |\w^{\delta}v_\varepsilon|^2\,\dd x\leq K_\varepsilon \int_{\RR^2}  |\w^{\delta_0}v_0|^2 \,\dd x+T K_\varepsilon \sqrt{N(u_0)} \,\|\D v \|_{C([0,T];L^2(\w^{-\delta'}))}\,.
\end{align*}
\end{lemma}
\begin{proof}
We have
\begin{align*}
&\frac{\dd}{\dd t}\int_{\RR^2}   |\w^{\delta} \,v_\varepsilon|^2 e^{-2Y_\varepsilon}\,\dd x=2\,\R{\int_{\RR^2} \w^{2\delta} \,  \partial_t v_\varepsilon \cdot\overline{v}_\varepsilon\, e^{-2Y_\varepsilon} \,\dd x}\\& =2\,\I{\int_{\RR^2} \w^{2\delta}(\Delta v_\varepsilon -2 \D v_\varepsilon \D Y_\varepsilon)\overline{v}_\varepsilon\,e^{-2Y_\varepsilon} \,\dd x} =2\,\I{\int_{\RR^2} \D\w^{2\delta}\cdot \D v_\varepsilon\,\overline{v}_\varepsilon e^{-2Y_\varepsilon}\,\dd x} \\
&\leq C \,\int_{\RR^2} \w^{2\delta-1} |\D v_\varepsilon| \,|v_\varepsilon| \,e^{-2Y_\varepsilon}\,\dd x\,.
\end{align*}
Integrating over $[0,T]$ then shows
\begin{align*}
\sup_{t\in [0,T]} \int_{\RR^2} |\w^{\delta}v_\varepsilon|^2 e^{-2Y_\varepsilon}\,\dd x\leq  \int_{\RR^2}  |\w^{\delta}v_0|^2 e^{-2Y_\varepsilon}+C T \sup_{t\in [0,T]} \int_{\RR^2} \w^{2\delta-1} |\D v_\varepsilon| \,|v_\varepsilon| \,e^{-2Y_\varepsilon} \,\dd x\,,
\end{align*}
so that the desired estimate follows with the Cauchy-Schwarz inequality and Corollary \ref{cor:BoundeY}. 
\end{proof}
We now derive an $H^1$ bound for $v_\varepsilon$. This is essentially based on an energy estimate, similar as done in \cite{SchroedingerTorus}, but using Lemma \ref{lem:H1GivesMoments} to controll appearing moments.
\begin{proposition}\label{prop:H1Bound}
Let $v_\varepsilon$ be the unique solution of \eqref{eq:MollifiedEquation} with $\lambda\leq 0$ or $\sigma<1$ on $[0,T]$, we then have for $\delta>0$ 
\begin{align*}
\|v_0\|_{H^1(\w^{-\delta})}\leq K_\varepsilon(1+\|v_0\|_{H^1(\w^{\delta_0})}^a)\,,
\end{align*}
for some deterministic $a>0$. 
\end{proposition}
\begin{proof}
Note first that $\|v\|_{C([0,T];L^2(\w^{-\delta})}\leq K_\varepsilon \|v_0\|^2_{L^2(\w^{\delta_0})}$ is clear by conservation of mass and Corollary \ref{cor:BoundeY}. Observe further that the claim follows if we can prove it for an arbitrarily small $\delta>0$. 

By the conservation of energy we obtain
\begin{align}
\label{eq:EnergyEstimate}
\int_{\RR^2} |\D v_\varepsilon|^2 e^{-2Y_\varepsilon} \dd x= 2\tilde{H}_\varepsilon(v_0)+\int_{\RR^2} \left( |v_\varepsilon|^2 \,\tDYe +\frac{\lambda}{\sigma+1} |v_\varepsilon |^{2\sigma+2} e^{-2\sigma Y_\varepsilon}  \right) e^{-2Y_\varepsilon}\,\dd x\,.
\end{align}
We estimate the first part of the integral on the right hand side by duality and Besov multiplication rules, \textit{(iii)} and \textit{(iv)} in Lemma \ref{lem:BesovProperties},
\begin{align*}
\int_{\RR^2} |v_\varepsilon|^2 \tDYe \overset{\textit{(iii)}}{\leq} C \||v_\varepsilon|^2\|_{\mathcal{B}^{\frac{1}{2}}_{1,1}(\w^{\delta})}\cdot  \|\tDYe\|_{\mathcal{C}^{-\frac{1}{2}}(\w^{-{\delta}})} 
\overset{\textit{(iv)}}{\leq} 
K_\varepsilon \|v_\varepsilon\|_{H^{\frac{1}{2}}(\w^{\delta/2})}^2\,.
\end{align*}
Now, using weighted interpolation (Lemma \ref{lem:BesovInterpolation}) and Lemma \ref{lem:H1GivesMoments} we have for $\delta$ small enough such that $2\delta<\delta_0$ and $\delta<1-2\delta$
\[
\|v_\varepsilon\|^2_{H^{\frac{1}{2}}(\w^{\delta/2})}\leq C \|v_\varepsilon\|_{L^2(\w^{2\delta})} \|v_\varepsilon\|_{H^1(\w^{-\delta})}\leq K_\varepsilon \|v_0\|_{L^2(\w^{\delta_0})} (1+\|v_\varepsilon\|_{H^1(\w^{-\delta})}^{3/2}) \,.
\]
Putting this into \eqref{eq:EnergyEstimate} and applying Young's inequality then yields
\begin{align}
\label{eq:EnergyEstimate2}
\|v_\varepsilon\|_{H^1(\w^{-\delta})}^2\leq C \tilde{H}_\varepsilon(v_0)+K_\varepsilon(1+\|v_0\|_{L^2(\w^{\delta_0})}^a)+\lambda C\int_{\RR^2} |v_\varepsilon|^{2\sigma+2} e^{-(2\sigma+2)Y_\varepsilon}\, \dd x\,.
\end{align}
If $\lambda\leq 0$ the last term is non-positive and can be dropped, otherwise consider the case $\sigma<1$.
Choose in the following $\kappa>0$ so small that $\sigma+\kappa/2<1$. Fix further $\overline{\delta}\in (0,\delta_0),\,\overline{\delta}'>0$ such that $\kappa \overline{\delta}-(1-\kappa)\overline{\delta}'>0$ and pick finally $\delta>0$ so small that we have both $\frac{\sigma}{\sigma+1}(-\delta)+\frac{1}{\sigma+1}(\kappa\overline{\delta}-(1-\kappa)\overline{\delta}')>0$ and $\delta<1-2\overline{\delta}$. We then have by Sobolev embedding $H^{\frac{\sigma}{\sigma+1}}\subseteq L^{2\sigma+2}$ (\textit{(ii)} in Lemma \ref{lem:BesovProperties}), Corollary \ref{cor:BoundeY}, Lemma \ref{lem:H1GivesMoments} and conservation of mass
\begin{align*}
\int_{\RR^2} |v_\varepsilon|^{2\sigma+2} e^{-(2\sigma+2)Y_\varepsilon}\, \dd x &\leq K_\varepsilon \|v_\varepsilon\|_{H^1(\w^{-\delta})}^{2\sigma}\,\|v_\varepsilon\|^2_{L^2(\w^{\kappa\overline{\delta}-(1-\kappa) \overline{\delta}'})} \\
&\leq K_\varepsilon \|v_\varepsilon\|_{H^1(\w^{-\delta})}^{2\sigma} \,\|v_\varepsilon\|^{2\kappa}_{L^2(\w^{\overline{\delta}})}\,\|v_\varepsilon\|^{2(1-\kappa)}_{L^2(\w^{- \overline{\delta}'})}\\
&\leq K_\varepsilon (1+\|v_0\|_{L^2(\w^{\delta_0})}^a) (1+\|v_\varepsilon\|_{H^1(\w^{-\delta})}^{2(\sigma+\kappa/2)})\,.
\end{align*}
Together with \eqref{eq:EnergyEstimate2} we get, by a further application of Young's inequality, the estimate
\begin{align*}
	\|v_\varepsilon\|_{C([0,T];H^1(\w^{-\delta})}\leq C\tilde{H}_\varepsilon(v_0)+K_\varepsilon\,(1+\|v_0\|_{L^2(\w^{\delta_0})}^a)\,,
\end{align*}
which implies the desired inequality. 
\end{proof}
Combining Proposition \ref{prop:H1Bound} and Lemma \ref{lem:H1GivesMoments} gives a uniform bound on the moments of $v_\varepsilon$.

\begin{corollary}\label{cor:MomentsBelow1}
If $v_\varepsilon$ is a solution to \eqref{eq:MollifiedEquation}  we have for $\gamma\in [0,1)$ and $\delta<(1-\gamma)\delta_0$
\begin{align*}
\|v_0\|_{H^{\gamma}(\w^{\delta})}\leq K_\varepsilon(1+\|v_0\|_{H^1(\w^{\delta_0})}^a)\,,
\end{align*}
for some deterministic $a>0$.
\end{corollary}
\begin{proof}
Inserting the estimate of Proposition \ref{prop:H1Bound} in Lemma \ref{lem:H1GivesMoments} we obtain for $\delta<\delta_0$
\begin{align*}
\|v_0\|_{L^2(\w^{\delta})}\leq K_\varepsilon(1+\|v_0\|_{H^1(\w^{\delta_0})}^a)\,.
\end{align*}
The result then follows by applying the interpolation inequality in Lemma \ref{lem:BesovInterpolation} with Proposition \ref{prop:H1Bound}.
\end{proof}
\section{Local existence}
\label{sec:LocalExistence}
To conclude existence a bound in $H^\gamma,\,\gamma>1$ is needed to make sense of the product term $-2\D v_\varepsilon \D Y_\varepsilon$ in the limit $\varepsilon\rightarrow 0$. In \cite{SchroedingerTorus} this was achieved by estimating the $L^2$ norm of the time derivative $w_\varepsilon= \frac{\dd}{\dd t} v_\varepsilon$, which morally corresponds to bounding the $H^2$ bound of $v_\varepsilon$. A key role in this estimate is the Brezis-Gallouet inequality \cite{BrezisGallouet} which is ill-suited for our approach based on polynomial weights. In Section \ref{sec:GlobalExistence} we show that if $\sigma<1/2$ this approach can be modified to yield global existence on $\RR^2$.  We here prove that one has local existence in time provided that one has either $\lambda\leq 0$ or $\sigma<1$
\begin{lemma}\label{lem:BlowUpH2}
Let $v_\varepsilon$ be the unique solution to \eqref{eq:MollifiedEquation} with $\lambda\leq 0$ or $\sigma<1$ on $[0,T]$. We then have for $\delta>0$
\begin{align*}
\|v_\varepsilon\|_{C([0,T];H^2(\w^{-\delta}))}\leq  K_\varepsilon (1+\|v_0\|_{H^2(\w^{\delta_0})}^a)\, e^{C T\|v_\varepsilon e^{-Y_\varepsilon}\|_{C([0,T];L^\infty(\RR^2))}^{2\sigma}} (1+|\log(\varepsilon)|^a)\,,
\end{align*}
for some deterministic $a>0$.
\end{lemma}
\begin{proof}
We consider as in \cite{SchroedingerTorus} the quantity $w_\varepsilon=\partial_t v_\varepsilon$ which satisfies the equation
\begin{align*}
\ii\partial_t w_\varepsilon=\Delta w_\varepsilon+w_\varepsilon\,\tDYe-2\D w_\varepsilon\D Y_\varepsilon+\lambda|v_\varepsilon e^{-Y_\varepsilon}|^{2\sigma} w_\varepsilon+\sigma\lambda v_\varepsilon|v_\varepsilon |^{2\sigma-2}\,2\,\R{w_\varepsilon\overline{v}_\varepsilon}\,e^{-2\sigma Y_\varepsilon}
\end{align*}
and whose mass evolves like 
\begin{align}
\label{eq:Derivew}
\frac{1}{2}\frac{\dd}{\dd t}\int_{\RR^2} |w_\varepsilon|^2 e^{-2Y_\varepsilon} \, \dd x &=2\sigma \lambda \int_{\RR^2}  \I{\overline{w}_\varepsilon v_\varepsilon} \R{w_\varepsilon \overline{v}_\varepsilon}\,|v_\varepsilon|^{2\sigma-2} e^{-(2\sigma+2)Y_\varepsilon}\dd x \\
&\leq C \|v_\varepsilon e^{-Y_\varepsilon}\|_{L^\infty}^{2\sigma} \int_{\RR^2} |w_\varepsilon|^2 e^{-2Y_\varepsilon}\, \dd x\,. \nonumber
\end{align}
Gronwall's lemma provides then
\begin{align}
\label{eq:H2BlowUp}
	\int_{\RR^2} |w_\varepsilon(t)|^2 e^{-2Y_\varepsilon}\leq \int_{\RR^2} |w(0)|^2 e^{-2Y_\varepsilon}\cdot e^{C T\|v_\varepsilon e^{-Y_\varepsilon}\|_{C([0,T];L^\infty(\RR^2))}^{2\sigma}} \,.
\end{align}
Recall that $\ii w_\varepsilon=\Delta v_\varepsilon+v_\varepsilon \, \tDYe -2\D v_\varepsilon \D Y_\varepsilon+\lambda |v_\varepsilon |^{2\sigma} v_\varepsilon e^{-2 \sigma Y_\varepsilon} $. By Sobolev embedding (\textit{(ii)} in Lemma \ref{lem:BesovProperties}) we have $\|v_0\|_{L^q(\w^{\delta})},\,\|\D v_0\|_{L^q(\w^{\delta})}\leq  C\|v_0\|_{H^2(\w^{\delta})}$ for $q>2,\,\delta\leq \delta_0$. Choose $q$ close enough to $2$ such that $q'$ with $\frac{1}{2}=\frac{1}{q}+\frac{1}{q'}$ satisfies $q'\cdot \delta_0>4$. We then have with Lemma \ref{lem:NoiseLp} 
\begin{align}
\|w_\varepsilon(0)\|_{L^2(\w^{\frac{\delta_0}{2}})} &\leq  C\|v_\varepsilon(0)\|_{H^2(\w^{\frac{\delta_0}{2}})}+\|v_\varepsilon(0)\|_{L^q(\w^{\delta_0/2})}\|\tDYe\|_{L^{q'(\w^{-\delta_0/2})}}\nonumber\\
&+2\|\D v_\varepsilon(0)\|_{L^q(\w^{\delta_0})} \|\D Y_\varepsilon\|_{L^{q'}(\w^{-\delta_0/2})} 
+\lambda \||v_\varepsilon(0)|^{2\sigma+1}\|_{L^q(\w^{\delta_0})} \|e^{-2\sigma Y_\varepsilon}\|_{L^{q'}(\w^{-\frac{\delta_0}{2}})}\nonumber\\
&\leq K_\varepsilon (1+\|v_0\|_{H^2(\w^{\delta_0})}^a)\cdot(1+|\log \varepsilon|^a)\,,
\label{eq:H2BlowUpRHS}
\end{align}
where we used Corollary \ref{cor:MomentsBelow1} and Sobolev embedding (\textit{(ii)} of Lemma \ref{lem:BesovProperties}) in the last step. 
On the other hand we have for $\delta>0$ and 
\begin{align*}
\|\Delta v_\varepsilon(t)\|_{L^2(\w^{-\delta})}&\leq \|w_\varepsilon(t)\|_{L^2(\w^{-\delta})}+\|v_\varepsilon(t)\|_{L^q(\w^{-3\delta/4})} \|\tDYe\|_{L^{q'}(\w^{-\delta/4})} \\&
+2\|\D v_\varepsilon(t)\|_{L^q(\w^{-3\delta/4})}\|\D Y_\varepsilon\|_{L^{q'}(\w^{-\delta/4})}\\
&+C\||v_\varepsilon(t)|^{2\sigma+1}\|_{L^q(\w^{-3\delta/4})} \|e^{-(2\sigma+1) Y_\varepsilon}\|_{L^{q'}(\w^{-\delta/4})} \,.
\end{align*}
where now $q\in (2,4)$ is small enough such that $q'$ with $\frac{1}{2}=\frac{1}{q}+\frac{1}{q'}$ satisfies $q'\cdot \delta>8$. By Sobolev embedding we have 
\[
\|v_\varepsilon(t)\|_{L^q(\w^{-3\delta/4})},\,\|\D v_\varepsilon(t)\|_{L^q(\w^{-3\delta/4})}\|\leq \|v_\varepsilon(t)\|_{H^{3/2}(\w^{-3\delta/4})}\leq \|v_\varepsilon(t)\|_{H^1(\w^{-\delta/2})}^{\frac{1}{2}} \,\|v_\varepsilon(t)\|_{H^{2}(\w^{-\delta})}^{\frac{1}{2}}\,.
\]
Applying Proposition \ref{prop:H1Bound} we therefore obtain for some $a>0$
\begin{align*}
\|\Delta v_\varepsilon(t)\|_{L^2(\w^{-\delta})}\leq \|w_\varepsilon(t)\|_{L^2(\w^{-\delta})}+K_\varepsilon (1+\|v_0\|_{H^2(\w^{\delta_0})}^a)(1+ |\log \varepsilon|^a) \,\|v_\varepsilon(t)\|_{H^2(\w^{-\delta})}^{1/2}\,.
\end{align*}
It is easy to see via \eqref{eq:PullWeight} that $\|g\|_{H^2(\w^{-\delta})}\leq C(\|g\|_{H^1(\w^{-\delta})}+\|\Delta g\|_{L^2(\w^{-\delta})})$ so that we obtain
\begin{align}
\label{eq:H2BlowUpLHS}
\|v_\varepsilon(t)\|_{H^2(\w^{-\delta})}\leq \|w_\varepsilon(t)\|_{L^2(\w^{-\delta})}+K_\varepsilon (1+\|v_0\|_{H^2(\w^{\delta_0})}^a)(1+ |\log \varepsilon|^a)\,.
\end{align}
Applying \eqref{eq:H2BlowUpLHS} to the left hand side and \eqref{eq:H2BlowUpRHS} to the right hand side of \eqref{eq:H2BlowUp} via Corollary \ref{cor:BoundeY} shows the desired estimate. 
\end{proof}
\begin{remark}
There was a technical subtlety in this proof which we hid from the reader for the sake of a clearer argument. Note that we do not know if the time derivative in \eqref{eq:Derivew} is well-defined. Due to the non-integer value of $\sigma$ it is not clear that we have even for smooth initial conditions smooth solutions which would allow for such an operation. For a rigoruos argument one really has to work instead with  $v_\varepsilon^{n}$, solution to $\ii \partial_t v^n_\varepsilon=\Delta v^n_\varepsilon +2 v^n_\varepsilon \tDYe- 2\D v_\varepsilon^n \D Y^\varepsilon +(|v^n_\varepsilon e^{-Y_\varepsilon}|^2+\frac{1}{n})^\sigma v^n_\varepsilon,\,\mathcal{S}(\RR^d)\ni v^n_\varepsilon(0)\rightarrow v_0$. One then easily derives bounds as above. Working with the backtransformed solution $u^n_\varepsilon=v^n_\varepsilon e^{-Y_\varepsilon}$ one can prove a $L^\infty(\RR^2)$ bound, uniform in $n$, on $v^\varepsilon_n e^{-Y_\varepsilon}$. One thus gets boundedness, uniform in $n$, of $\|v^n_\varepsilon\|_{H^2(\w^{-\delta})}$. By choice of a weakly convergent subsequence and the compact embedding $H^{\gamma}(\w^\delta)\subseteq H^{\gamma'} (\w^{\delta'}),\,\gamma'<\gamma,\delta'<\delta$ one concludes. 
\end{remark}

This estimate is sufficient to prove local existence. We follow \cite{SchroedingerTorus} in the consideration of the differences of the dyadic sequence. Recall that the notation $K_k$ in the following stands for a random constant of the form $K_k=K_{2^{-k}}K_{2^{-k-1}}$ that can be bounded almost surely in $k$. To derive the latter property we will always use, without mentioning, the results in subsection \ref{subsec:Noise}.
\begin{lemma}
\label{lem:ConvergenceDyadicSequence}
Let $\vk$ be the unique solution to \eqref{eq:MollifiedEquation} on $[0,T]$ with $\varepsilon=2^{-k}$ and $\lambda\leq 0$ or $\sigma<1$. We then have for $\gamma\in (0,2),\,\delta<(1-\frac{\gamma}{2})\delta_0$
\begin{align*}
&\|\vk-\vkp\|_{C([0,T];H^\gamma(\w^\delta))}\leq \\ & K_{k} 2^{-k\kappa} (1+\|v_0\|_{H^2(\w^{\delta_0})}^a)\,e^{CT (\|\vk e^{-\Yk}\|^{2\sigma}_{C([0,T];L^\infty)}+\|\vkp e^{-\Ykp}\|^{2\sigma}_{C([0,T];L^\infty)})}\,.
\end{align*}
for some $\kappa>0$ and $a>0$, where the sequence of random constants $K_{k}$ is bounded almost surely. 
\end{lemma}
\begin{proof}
The difference $r_k=\vk-\vkp$ satisfies the equation
\begin{align*}
\ii \partial_t r_k&=\Delta r_k +r_k \, \widetilde{\Wick{\D Y_{2^{-k-1}}^2}}-2\D r_k \D Y_{2^{-k-1}}+\vk (\widetilde{\Wick{\D Y_{2^{-k-1}}^2}}-\widetilde{\Wick{\D Y_{2^{-k}}^2}}) \\
&-2 \D \vk(\D Y_{2^{-k-1}}-\D Y_{2^{-k}})+\lambda (|\vk e^{- Y_{2^{-k}}}|^{2\sigma}\vk -|\vkp e^{-Y_{2^{-k-1}}}|^{2\sigma}\vkp )\,,
\end{align*}
so that the ``mass'' of $r_k$ evolves according to 
\begin{align*}
\frac{\dd}{\dd t} \int_{\RR^2} |r_k|^2 e^{-2\Ykp}&=\mathrm{Im}\Bigg\{\int_{\RR^2} \Big(
\vk (\widetilde{\Wick{\D \Ykp^2}}-\widetilde{\Wick{\D \Yk^2}}) \overline{r}_k e^{-2\Ykp} \\
&-2 \D \vk(\D \Ykp-\D \Yk) \overline{r}_k e^{-2\Ykp}\\
&+\lambda (|\vk e^{- Y_{2^{-k}}}|^{2\sigma}\vk -|\vkp e^{-Y_{2^{-k-1}}}|^{2\sigma}\vkp )\overline{r}_k e^{-2\Ykp}\Big)\,\dd x\Bigg\}\,.
\end{align*}
Via \textit{(iii)}, \textit{(iv)} of Lemma \ref{lem:BesovProperties} (Duality and Multiplication bound)  and using interpolation between the bound in Corollary \ref{cor:MomentsBelow1} and Lemma \ref{lem:BlowUpH2} we can estimate the first two terms on the right hand side, up to a constant, by
\begin{align*}
&\|\vk \overline{r}_k e^{-2\Ykp} \|_{\mathcal{B}^{\frac{1}{2}}_{1,1}(\w^{\delta'})}\,\|\widetilde{\Wick{\D \Ykp^2}}-\widetilde{\Wick{\D \Yk^2}}\|_{\mathcal{C}^{-\frac{1}{2}}(\w^{-\delta')})} \\
&+\|\D \vk \overline{r}_k e^{-2\Ykp}\|_{\mathcal{B}^{\frac{1}{2}}_{1,1}(\w^{\delta'})}\,\|(\D \Ykp-\D \Yk)\|_{\mathcal{C}^{-\frac{1}{2}}(\w^{-\delta'})}\\ 
&\leq K_k (\|\vk\|_{H^{\frac{3}{2}}(\w^{\delta'})}^2+\|\vkp\|_{H^{\frac{3}{2}}(\w^{\delta'})}^2)\,2^{-k\kappa'}\\
& \leq (1+\|v_0\|_{H^2(\w^{\delta_0})}^a) K_{k} 2^{-k\kappa'/2}\,e^{CT (\|\vk\|^{2\sigma}_{C([0,T];L^\infty)}+\|\vkp\|^{2\sigma}_{C([0,T];L^\infty)})}\,,
\end{align*}
for $\delta'\in (0,\delta_0/4)$ and $\kappa'<1/2$.
Up to a term $K_k 2^{-k\kappa} (1+\|v_0\|_{H^1(\w^{\delta_0})}^a)$ we can reshape the third term as
\begin{align*}
&\mathrm{Im}\left\{\int_{\RR^2} \lambda (|\vk e^{-\Yk}|^{2\sigma}\vk e^{- \Yk}-|\vkp e^{-\Ykp}|^{2\sigma}\vkp e^{- \Ykp})\overline{r}_k e^{-\Ykp} \, \dd x\right\}\,.
\end{align*}
Recall that for $x,y\in\mathbb{C}$ $||x|^{2\sigma}x-|y|^{2\sigma}y|\leq C (|x|^{2\sigma}+|y|^{2\sigma})\,|x-y|$ (see for example \cite[p. 86]{Cazenave}) so that we obtain the upper bound
\begin{align*}
\int_{\RR^2} (|\vk e^{-\Yk}|^{2\sigma}+|\vkp e^{-\Ykp}|^{2\sigma}) |\vk e^{-\Yk}-\vkp e^{-\Ykp}| \overline{r}_k e^{-\Ykp} \, \dd x \,,
\end{align*}
Applying \eqref{eq:ConvergenceeY} and Corollary \ref{cor:MomentsBelow1} we can bound this up to a term $ K_k (1+\|v_0\|_{H^1(\w^{\delta_0})}^a) 2^{-k\kappa} $ by 
\begin{align*}
&\int_{\RR^2} (|\vk e^{-\Yk}|^{2\sigma}+|\vkp e^{-\Ykp}|^{2\sigma}) |r_k|^2 e^{-2\Ykp} \, \dd x \\
 &\leq (\|\vk e^{-\Yk}\|_{L^\infty}^{2\sigma}+\|\vkp e^{-\Ykp}\|_{L^\infty}^{2\sigma}) \int_{\RR^2} |r_k|^2 e^{-2\Ykp} \, \dd x\,.
\end{align*} 
Putting everything together we have
\begin{align*}
\frac{\dd}{\dd t} \int_{\RR^2} |r_k|^2 e^{-2\Ykp} &\leq K_{k} 2^{-k\kappa}\,e^{CT (\|\vk\|^{2\sigma}_{C([0,T];L^\infty)}+\|\vkp\|^{2\sigma}_{C([0,T];L^\infty)})} \\ 
&+C(\|\vk e^{-\Yk}\|_{L^\infty}^{2\sigma}+\|\vkp e^{-\Ykp}\|_{L^\infty}^{2\sigma}) \int_{\RR^2} |r_k|^2 e^{-2\Ykp} \, \dd x\,,
\end{align*}
for some $\kappa>0$. Application of Gronwall's lemma gives together with Corollary \ref{cor:BoundeY}
\begin{align*}
&\|\vk-\vkp\|_{C([0,T];L^2(\w^{-\delta}))}\leq \\ & K_{k} 2^{-k\kappa} (1+\|v_0\|_{H^2(\w^{\delta_0})})\,e^{CT (\|\vk e^{-\Yk}\|^{2\sigma}_{C([0,T];L^\infty)}+\|\vkp e^{-\Ykp}\|^{2\sigma}_{C([0,T];L^\infty)})}\,.
\end{align*}
for any $\delta>0$. The desired estimate now follows by interpolation with Corollary \ref{cor:MomentsBelow1} and Lemma \ref{lem:BlowUpH2}.
\end{proof}

We are now in the position to prove local existence.
\begin{theorem}\label{thm:LocalExistence}
There is a (random) time $T>0$ such that \eqref{eq:TransformedEquation} with $\lambda\leq 0$ or $\sigma<1$ has almost surely a (unique) solution $v$ in $C([0,T];H^{\gamma}(\w^\delta))$ for $\gamma\in (1,2)$ and $\delta<(1-\frac{\gamma}{2})\delta_0$ and such that the random variable $\|v_\varepsilon -v\|_{C([0,T];H^{\gamma}(\w^\delta)}$ converges to 0 in probability. 
\end{theorem}
\begin{proof}
Let $M_T^N:=\sup_{k\leq N} \|\vk\|_{C([0,T];H^\gamma(\w^\delta))}$, we then have summing the estimate in Lemma \ref{lem:ConvergenceDyadicSequence} 
\begin{align*}
M_T^N &\leq \|v_1\|_{C([0,T];H^\gamma(\w^\delta))}+ \sum_{k=1}^\infty K_k 2^{-k\kappa} (1+\|v_0\|_{H^2(\w^{\delta_0})}) e^{CT M^N_T} \\
&\leq K(1+\|v_0\|_{C([0,T];H^2(\w^\delta))}) e^{CT M^N_T}\,,
\end{align*} 
where the random constant $K$ is finite almost surely and moreover, by Minkowski's inequality, in any $L^p(\mathbb{P}),\, p\in [1,\infty)$. Consequently there is a random time $T$, independent of $N$ such that for any $N\in \mathbb{N}$
\begin{align*}
M^N_T \leq 2 K(1+\|v_0\|_{C([0,T];H^2(\w^\delta))})
\end{align*}
and thus we have that $\|\vk\|_{C([0,T];H^\gamma(\w^\delta))}$ is uniformly bounded. Reinserting this in Lemma \ref{lem:ConvergenceDyadicSequence} shows that $\vk$ is a Cauchy sequence and thus we can conclude convergence to a $v$ that solves \eqref{eq:TransformedEquation}. Uniqueness of the solution follows by standard arguments.

 The convergence of probablity follows similar as in \cite{SchroedingerTorus} by considering first $\|v_\varepsilon -\vk\|_{C([0,T];H^{\gamma}(\w^\delta))}$: Redoing then the proof of Lemma \ref{lem:ConvergenceDyadicSequence} but bounding $\|\vk\|_{H^{3/2}(\w^\delta)}$, $\|\vk e^{-\Yk}\|_{L^\infty}$ directly instead of applying Lemma \ref{lem:BlowUpH2} we can let $k\rightarrow \infty$ and the resulting estimate yields the convergence in probability.  
\end{proof}

\section{Global existence for $\sigma<1/2$}
\label{sec:GlobalExistence}
In the case $\sigma<1/2$ we now prove that there is a global solution to \eqref{eq:TransformedEquation}. Let us first modify the famous Brezis-Gallouet inequality for our purposes
\begin{lemma}\label{lem:BrezisGallouet}
For the solution $v_\varepsilon$ of \eqref{eq:TransformedEquation} with $\lambda\leq 0$ or $\sigma<1$ we have for $\gamma\in (1,2)$ and $\kappa>0$ and $\delta<(\gamma-1)\delta_0$
\begin{align*}
\|v_\varepsilon e^{-Y_\varepsilon}\|_{C([0,T];L^\infty(\RR^2))}\leq K_\varepsilon +(1+|\log \varepsilon|^{1+\kappa})(1+\|v_0\|_{H^1(\w^{\delta_0})}^a)+\log(1+ \|v_\varepsilon\|_{C([0,T];H^{\gamma}(\w^{-\delta})})\,,
\end{align*}
for some deterministic $a>0$
\end{lemma}
\begin{proof}
Choose $\delta_0'\in (0,\delta_0)$ and $\gamma'\in (1,\gamma)$ large enough so that we have $\delta<\frac{\gamma-\gamma'}{\gamma'}\delta_0'$ and thus $\delta':=\frac{\gamma-\gamma'}{\gamma}\delta_0'-\frac{\gamma'}{\gamma}\delta\in (0,\delta_0)$. We use the Brezis-Gallouet inequality as proved in \cite{Ozawa} which gives us
\begin{align*}
\|v_\varepsilon e^{-Y_\varepsilon}\|_{L^\infty(\RR^2)}&\leq (1+\|v_\varepsilon e^{-Y_\varepsilon}\|_{H^1(\RR^2)}) \,\sqrt{1+\log(1+\|v_\varepsilon e^{-Y_\varepsilon}\|_{\mathcal{C}^{\gamma-1}(\RR^2)})} \\
&\leq (1+\|v_\varepsilon e^{-Y_\varepsilon}\|_{H^1(\RR^2)})\sqrt{1+\log(1+K_\varepsilon \|v_\varepsilon\|_{\mathcal{C}^{\gamma'-1}(\w^{\delta'})})} \\
&\leq K_\varepsilon +\|v_\varepsilon e^{-Y_\varepsilon}\|_{H^1(\RR^2)}^2+\log(1+\|v_\varepsilon\|_{H^{\gamma'}(\w^{\delta'})})\\
&\leq K_\varepsilon +\|v_\varepsilon e^{-Y_\varepsilon}\|_{H^1(\RR^2)}^2+\log(1+\|v_\varepsilon\|_{H^{\gamma}(\w^{-\delta})}^{\frac{\gamma'}{\gamma}}\|v_\varepsilon|_{L^2(\w^{\delta_0'})}^{\frac{\gamma-\gamma'}{\gamma}})\,.
\end{align*}
Note that we have by the product rule
\begin{align*}
\|v_\varepsilon e^{-Y_\varepsilon}\|_{H^1(\RR^2)}\leq  C(\|v_{\varepsilon} e^{-Y_\varepsilon}\|_{L^2(\RR^2)}+\|\D v_\varepsilon e^{-Y_\varepsilon}\|_{L^2(\RR^2)}+\|v_\varepsilon e^{-Y_\varepsilon} \D Y_\varepsilon \|_{L^2(\RR^2)})\,.
\end{align*}
While the first term is bounded by conservation of mass, the second can be bounded by conservation of energy \eqref{eq:EnergyEstimate}. For the last we apply H\"older's inequality and Lemma \ref{lem:NoiseLp} to bound it by
\begin{align}
\label{eq:BrezisBoundH1}
\|v_\varepsilon e^{-Y_\varepsilon} \D Y_\varepsilon \|_{L^2(\RR^2)}\leq \|v_\varepsilon e^{-Y_\varepsilon}\|_{L^q(\w^{\tilde{\delta}})} \|\D Y_\varepsilon\|_{L^{q'}(\w^{-\tilde{\delta}})} \leq K_\varepsilon (1+\|v_0\|_{H^1(\w^{\delta_0})}^a) \,|\log\varepsilon|^{1+\kappa/2}\,,
\end{align}
for some small $\tilde{\delta}>0$ and $\frac{1}{2}=\frac{1}{q}+\frac{1}{q'}$ with $q'\in(2,\infty)$ large enough so that $\tilde{\delta}\cdot q'>2$ and where we used Corollary \ref{cor:MomentsBelow1} and \ref{lem:NoiseLp} in the second step. Note that the constants $K_\varepsilon$ are bounded almost surely for dyadic $\varepsilon=2^{-k}$ by Lemma \ref{lem:NoiseLp} and Borel-Cantelli. 

Inserting \eqref{eq:BrezisBoundH1} above shows the desired inequality. 
\end{proof}

The combination of Lemma \ref{lem:BrezisGallouet} with \ref{lem:BlowUpH2} and \ref{lem:ConvergenceDyadicSequence} gives global existence if $\sigma<\frac{1}{2}$.

\begin{theorem}\label{thm:GlobalExistence}
The equation \eqref{eq:TransformedEquation} with $\lambda=0$ or $\sigma\in(0,1/2)$ has for every $T>0$ a unique solution $v\in C([0,T];H^\gamma(\w^{\delta}))$ for $\gamma\in (1,2)$ and $\delta<(1-\frac{\gamma}{2})\delta_0$. The solutions $v_\varepsilon$ of \eqref{eq:MollifiedEquation} converge in probability in $C([0,T];H^{\gamma}(\w^\delta))$ to $v$.
\end{theorem}
\begin{proof}
We assume $\sigma\in (0,1/2)$, since the linear case $\lambda=0$ is trivially included in this range.  
Inserting the Brezis-Gallouet inequality in Lemma \ref{lem:BrezisGallouet} in the bound in Lemma \ref{lem:BlowUpH2} gives for some $a>0$
\begin{align*}
\|\vk\|_{H^{\gamma}(\w^\delta)}&\leq C (1+\|v_0\|_{H^2(\w^{\delta_0})}^a)\,\|\vk\|_{H^{2}(\w^{-\delta'})}^{\frac{\gamma}{2}}
 \\
 &\leq e^{K_k} e^{(1+\|v_0\|_{H^2(\w^{\delta_0})}^a)(1+|\log \varepsilon|)^{2\sigma(1+\kappa)}} e^{\log(1+\|v\|_{H^{\gamma}(\w^\delta)})^{2\sigma}}  \,,
\end{align*}
for some small $\delta'>0$ and $\kappa>0$. Note that for $s\in (0,1)$ and any $\kappa'>0$ $e^{\log(1+x)^s}\leq C x^{\kappa'}$, so that we can reformulate this estimate as
\begin{align*}
\|\vk\|_{H^{\gamma}(\w^\delta)}&\leq e^{K_k} e^{(1+\|v_0\|_{H^2(\w^{\delta_0})}^a)(1+|\log \varepsilon|)^{2\sigma(1+\kappa)}}
\end{align*}
Reinserting this into the Brezis-Gallouet inequality, Lemma \ref{lem:BrezisGallouet}, yields
\begin{align*}
\|v_\varepsilon e^{-Y_\varepsilon}\|_{C([0,T];L^\infty(\RR^2))}\leq K_\varepsilon +(1+|\log \varepsilon|^{1+\kappa})(1+\|v_0\|_{H^2(\w^{\delta_0})}^a)\,,
\end{align*}
where we choose $\kappa>0$ small enough such that $s:=2\sigma(1+\kappa)<1$.
Combining this estimate with Lemma \ref{lem:ConvergenceDyadicSequence} we end up with 
\begin{align*}
&\|\vk-\vkp\|_{C([0,T];H^\gamma(\w^\delta))}\leq   2^{-k\kappa'} e^{K_k} e^{(1+\|v_0\|_{H^2(\w^{\delta_0})})(1+|\log 2^{-k}|)^{s}}\leq C(v_0) e^{K_k} 2^{-k \kappa'/2}\,.\,
\end{align*}
for some $\kappa'>0$ and where we used $s<1$ in the second step. We therefore conclude that $\vk$ is a Cauchy sequence whose limit $v\in C([0,T];H^{\gamma}(\w^\delta))$ solves \eqref{eq:TransformedEquation}. For the convergence of $v_\varepsilon$ in probability to $v$ we proceed as in Theorem \ref{thm:LocalExistence}.
\end{proof}

{\footnotesize
\bibliography{nls-Ref}

\begin{thebibliography}{BCD11}

\bibitem[AC15]{ChoukAllez}
R.~{Allez} and K.~{Chouk}.
\newblock {The continuous Anderson hamiltonian in dimension two}.
\newblock {\em ArXiv e-prints}, November 2015.

\bibitem[And58]{AndersonLocalization}
P.~W. Anderson.
\newblock Absence of diffusion in certain random lattices.
\newblock {\em Phys. Rev.}, 109:1492--1505, Mar 1958.

\bibitem[BCD11]{Bahouri}
Hajer Bahouri, Jean-Yves Chemin, and Rapha{\"e}l Danchin.
\newblock {\em Fourier Analysis and Nonlinear Partial Differential Equations}.
\newblock Springer Berlin Heidelberg, Berlin, Heidelberg, 2011.

\bibitem[Ber98]{Berge}
Luc Bergé.
\newblock Wave collapse in physics: principles and applications to light and
  plasma waves.
\newblock {\em Physics Reports}, 303(5):259 -- 370, 1998.

\bibitem[BG80]{BrezisGallouet}
H.~Brezis and T.~Gallouet.
\newblock Nonlinear schrödinger evolution equations.
\newblock {\em Nonlinear Analysis: Theory, Methods \& Applications}, 4(4):677
  -- 681, 1980.

\bibitem[Caz79]{Cazenave1979}
Thierry Cazenave.
\newblock Equations de schr{\"o}dinger non lin{\'e}aires en dimension deux.
\newblock {\em Proceedings of the Royal Society of Edinburgh Section A:
  Mathematics}, 84(3-4):327--346, 1979.

\bibitem[Caz03]{Cazenave}
T.~Cazenave.
\newblock {\em Semilinear Schr{\"o}dinger Equations}.
\newblock Courant lecture notes in mathematics. American Mathematical Society,
  2003.

\bibitem[Con12]{Conti}
Claudio Conti.
\newblock Solitonization of the anderson localization.
\newblock {\em Phys. Rev. A}, 86:061801, Dec 2012.

\bibitem[DW16]{SchroedingerTorus}
A.~{Debussche} and H.~{Weber}.
\newblock The schr{\"o}dinger equation with spatial white noise potential.
\newblock {\em ArXiv e-prints}, December 2016.

\bibitem[FC10]{ContiFolli}
Viola Folli and Claudio Conti.
\newblock Frustrated brownian motion of nonlocal solitary waves.
\newblock {\em Physical Review Letters}, 104:193901, May 2010.

\bibitem[Hai14]{RegularityStructures}
M.~Hairer.
\newblock A theory of regularity structures.
\newblock {\em Inventiones mathematicae}, 198(2):269--504, Nov 2014.

\bibitem[HL15]{HairerLabbe}
Martin Hairer and Cyril Labbé.
\newblock A simple construction of the continuum parabolic anderson model on
  $\mathbf{R}^2$.
\newblock {\em Electron. Commun. Probab.}, 20:11 pp., 2015.

\bibitem[Kat87]{Kato}
Tosio Kato.
\newblock On nonlinear schrödinger equations.
\newblock {\em Annales de l'I.H.P. Physique théorique}, 46(1):113--129, 1987.

\bibitem[Oza95]{Ozawa}
T.~Ozawa.
\newblock On critical cases of sobolev's inequalities.
\newblock {\em Journal of Functional Analysis}, 127(2):259 -- 269, 1995.

\bibitem[PT16]{PromelTrabs}
David~J. Prömel and Mathias Trabs.
\newblock Rough differential equations driven by signals in besov spaces.
\newblock {\em Journal of Differential Equations}, 260(6):5202 -- 5249, 2016.

\bibitem[SSV14]{Sickel}
Winfried Sickel, Leszek Skrzypczak, and Jan Vyb{\'i}ral.
\newblock Complex interpolation of weighted besov and lizorkin-triebel spaces.
\newblock {\em Acta Mathematica Sinica, English Series}, 30(8):1297--1323, Aug
  2014.

\bibitem[Tri83]{TriebelI}
Hans Triebel.
\newblock {\em Theory of Function Spaces}.
\newblock Springer Basel, Basel, 1983.

\bibitem[Tri06]{TriebelIII}
{\em Theory of Function Spaces III}.
\newblock Birkh{\"a}user Basel, Basel, 2006.

\end{thebibliography}
\bibliographystyle{alpha}}
\end{document}